\documentclass[preprint,3p, 11pt]{elsarticle}

\usepackage{amssymb}
\usepackage{amsthm, amsmath}

\journal{Discrete Applied Mathematics} 

\def\N{\mathbb N}
\def\uu{\mathbf u}
\def\vv{\mathbf v}
\def\tt{\mathbf t}
\def\Der{\mathrm{d}}
\def\nrc{nr\mathcal{C}}

\begin{document}

\newtheorem{theorem}{Theorem}
\newtheorem{corollary}[theorem]{Corollary}
\newtheorem{lemma}[theorem]{Lemma}
\newtheorem{claim}[theorem]{Claim}
\newtheorem{proposition}[theorem]{Proposition}
\newtheorem{definition}[theorem]{Definition}
\newtheorem{example}[theorem]{Example}

\begin{frontmatter}


\title{On non-repetitive complexity of Arnoux--Rauzy words\tnoteref{funding}}
\tnotetext[funding]{The research received funding from the Ministry of Education, Youth and Sports of the Czech Republic through the project no. CZ.02.1.01/0.0/0.0/16 019/0000778 and from the Czech Technical University in Prague through the project SGS17/193/OHK4/3T/14. The last author acknowledges partial funding via a Welcome Grant of the University of Liege.}

\author[FNSPE]{Kate\v{r}ina Medkov\'{a}\corref{corA}}
 \ead{katerinamedkova@gmail.com}
 \cortext[corA]{Corresponding author}
\author[FNSPE]{Edita Pelantov\'a}
\author[LIEGE]{\'Elise Vandomme}

\address[FNSPE]{
Department of Mathematics, FNSPE, Czech Technical University in Prague,
Trojanova 13, 120 00 Praha 2, Czech Republic}

\address[LIEGE]{
Department of Mathematics, University of Liege,
Allée de la découverte 12, 4000 Liege, Belgium}

\begin{abstract}
The non-repetitive complexity $nr\mathcal{C}_{\bf u}$ and the initial non-repetitive complexity $inr\mathcal{C}_{\bf u}$ are functions which reflect  the  structure of the infinite word ${\bf u}$ with respect to the repetitions of factors of a given length. 
We  determine $nr\mathcal{C}_{\bf u}$  for the Arnoux--Rauzy words and $inr\mathcal{C}_{\bf u}$ for the  standard  Arnoux--Rauzy words. Our main tools   are $S$-adic representation of Arnoux--Rauzy words and description of return words to their  factors. The formulas we obtain are then used   to evaluate  $nr\mathcal{C}_{\bf u}$ and  $inr\mathcal{C}_{\bf u}$  for the $d$-bonacci word.
\end{abstract}

\end{frontmatter}


\section{Introduction} 

Variability of an infinite word ${\bf u} = u_0u_1u_2\cdots$ over a finite alphabet can be judged from distinct points of view depending on applications or combinatorial properties one is interested in.  
The factor complexity of ${\bf u} $, here denoted $\mathcal{C}_{\bf u}$, is a function which to any $n\in \mathbb{N}$ assigns the number 
 of distinct factors of length $n$ occurring in $\uu$.  
More formally, $\mathcal{C}_{\bf u}(n) = \#\{ u_iu_{i+1}\cdots u_{i+n-1}\ :\ i \in \mathbb{N}\}$. 

For the simplest infinite words, namely the eventually periodic words, the factor complexity is bounded from above by a constant. 
In \cite{MoHe40}, Morse and Hedlund showed that the factor complexity of an infinite word which is not eventually periodic satisfies $\mathcal{C}_{\bf u}(n)\geq n+1 $ for each $n \in \mathbb{N}$. 
If the equality takes place for each $n$, the word ${\bf u}$ is called Sturmian.
Sturmian words represent the most intensively studied class of infinite words.
To measure the regularity of an infinite word, Morse and Hedlund introduced the recurrence function $R_{\bf u}$.
The value $R_{\bf u}(n)$ is defined to be the minimal integer $m$ such that any factor of $\uu$ of length $n$ occurs at least once in $u_{i}u_{i+1}u_{i+2}\cdots u_{i+m-1}$ for every $i \in \mathbb{N}$.
In the same paper \cite{MoHe40}, the authors evaluated $R_{\uu}(n)$ for any Sturmian word.   
  
A dual function to $R_{\bf u}$ was recently introduced by Moothathu \cite{Mo12} under the name non-repetitive complexity function $nr\mathcal{C}_{\uu}$. 
The value $nr\mathcal{C}_{\uu}(n)$ is defined as the maximal $m$ such that for some $i \in \mathbb{N}$ any factor of $\uu$ of length $n$ occurs at most once in    $u_{i}u_{i+1}u_{i+2}\cdots u_{i+m+n-2}$. 
He also considered a ``prefix variant'' of this function called the initial non-repetitive complexity function $inr\mathcal{C}_{\uu}$. 
By definition, $inr\mathcal{C}_{\uu}(n)$ is the maximal length $m$ of a prefix of $\uu$ such that each factor of $\uu$ of length $n$ occurs in $u_0u_1\cdots u_{m+n-2}$ at most once. 
Obviously, 
$$ inr\mathcal{C}_{\uu}(n) \ \leq \ nr\mathcal{C}_{\uu}(n)\ \leq \   \mathcal{C}_{\uu}(n)\  \leq \  R_{\uu}(n) -n+1 \quad \text{for each } n \in\mathbb{N}. $$
Moothathu's concept of the initial non-repetitive complexity function was developed in \cite{NiRa16} by  Nicholson and Rampersad. 
They described some general properties of $inr\mathcal{C}_{\uu}$ and evaluated $inr\mathcal{C}_{\uu}$ for  the Fibonacci, Tribonacci and Thue--Morse words.
Note that the Fibonacci word and the Tribonacci word belong to the class of standard binary and ternary, respectively, Arnoux--Rauzy words.
The Arnoux--Rauzy words represent one of the generalizations of Sturmian words to multi-letter alphabets.
The recurrence function $R_{\uu}$ for Arnoux--Rauzy words was determined  in \cite{CaCh}. 
The initial  non-repetitive complexity function for Sturmian sequences was recently studied  by Bugeaud and Kim \cite{BuKi19}.  Their motivation for this study comes from the connection between the irrational exponent of a number $x$   and $inr\mathcal{C}_{\uu}$, where ${\uu}$ corresponds to the expansion of $x$ in a given base.

In the present article we focus on the non-repetitive complexity of Arnoux--Rauzy words.
Using the $S$-adic representation of a given Arnoux--Rauzy word $\uu$, we provide  in Theorem  \ref{thm:nonrepetitive_complexity_arnoux_rauzy} a formula for computing $nr\mathcal{C}_{\uu}(n)$ for each $n \in \mathbb{N}$. 
In particular, we show  (Theorem \ref{thm:nonrepetitive_complexity_sturmian}) that any Sturmian word (i.e., binary Arnoux--Rauzy word) $\uu$ satisfies $ nr\mathcal{C}_{\uu}(n) = \mathcal{C}_{\uu}(n) $ for each $n \in \mathbb{N}$.
It is interesting that this phenomenon can be observed also among the words with the maximal factor complexity.
In \cite{NiRa16}, the authors constructed a word over $q$ letter alphabet such that $q^n  =\mathcal{C}_{\uu}(n) = nr\mathcal{C}_{\uu}(n)$. 

For standard Arnoux--Rauzy words we determine in Theorem \ref{thm:bonacci_initial_nonrepetitive_complexity} also $inr\mathcal{C}_{\uu}$ and thus we gene\-ralize Nicholson and Rampersad's result on the Fibonacci and the Tribonacci words.

\section{Preliminaries} \label{sec:preliminaries}

An \emph{alphabet} $\mathcal{A}$ is a finite set of symbols called \emph{letters}.
Here we fix the alphabet $\mathcal{A} = \{0, 1, \ldots, d-1\}$, where $d$ is a positive integer.
A \emph{word} $w=w_0\cdots w_{n-1}$ over $\mathcal{A}$ is a finite sequence of letters from $\mathcal{A}$.
The number of its letters is called the \emph{length} of $w$ and it is denoted by $|w|=n$. 
The notation $|w|_a$ is used for the number of
occurrences of the letter $a$ in $w$. The \emph{empty word}, i.e., the unique word of length zero, is denoted by $\varepsilon$. 
The \emph{concatenation} of words $v=v_0\cdots v_{k}$ and $w=w_0\cdots w_{\ell}$ is the word $vw=v_0\cdots v_{k}w_0\cdots w_\ell$. 
The set of all finite words over $\mathcal{A}$ equipped with the operation concatenation of words is a free monoid and it is denoted $\mathcal{A}^*$. 
The \emph{Parikh vector} of a word $w\in \mathcal{A}^*$ is the vector $\vec{V}(w)=(|w|_{0},  |w|_{1}, \ldots, |w|_{d-1})^\top$.
Obviously, $|w|=(1,1, \cdots, 1) \cdot \vec{V}(w)$.

An infinite sequence of letters $\uu=(u_i)_{i\geq 0}$ in $\mathcal{A}$ is called \emph{infinite word}. 
The set of all infinite words over $\mathcal{A}$ is denoted $\mathcal{A}^{\N}$. 
The word $\uu\in \mathcal{A}^{\N}$ is said to be \emph{eventually periodic} if it is of the form $\uu=vz^{\omega}$, where $v,z\in \mathcal{A}^*$, $z\neq\varepsilon$ and $z^{\omega}=zzz\cdots$.  
Otherwise, $\uu$ is \emph{aperiodic}.    

A \emph{factor} of a (finite of infinite) word $w$ is a finite word $v$ such that $w=svt$ for some 
words $s,t \in \mathcal{A}^*$. 
Moreover, if $s=\varepsilon$, then $v$ is called a \emph{prefix} of $w$ and if $t=\varepsilon$, then $v$ is called a \emph{suffix} of $w$. 
The set of all factors of an infinite word $\uu$ is called the \emph{language} of $\uu$ and denoted by $\mathcal{L}_\uu$.
By $\mathcal{L}_\uu(n)$ we denote the set of factors of $\uu$ of length $n$, i.e., $\mathcal{L}_\uu(n) =\mathcal{L}_\uu \cap \mathcal{A}^n $.
Using this notation, the factor complexity of $\uu$ can be expressed as $\mathcal{C}_{\uu}(n) = \#\mathcal{L}_\uu(n) $ for every $n \in \N$. 
In this paper, we focus on the (initial) non-repetitive complexity.

\begin{definition}
The \emph{non-repetitive complexity} $\nrc_\uu$ and the \emph{initial non-repetitive complexity} $i\nrc_\uu$ of an infinite word $\uu$ are functions defined for each $n \in \N$ as follows
$$\nrc_\uu(n):=\max\{m\in\N: \exists k\in\N\text{ s.t. } u_i\cdots u_{i+n-1}\neq u_j\cdots u_{j+n-1}\ \forall i,j \text{ with }k\le i<j\le k+m-1\},$$
$$i\nrc_\uu(n):=\max\{m\in\N: u_i\cdots u_{i+n-1}\neq u_j\cdots u_{j+n-1}\ \forall i,j \text{ with }0\le i<j\le m-1\}\,.$$
\end{definition}

A factor $w$ of $\uu$ is \emph{right special} if there exist two distinct letters $a,b \in \mathcal{A}$ such that $wa$ and $wb$ belong to  $\mathcal{L}_\uu$.  
Analogously, $w$ is \emph{left special} if $aw$ and $bw$ belong to $\mathcal{L}_\uu$ for two distinct letters $a,b \in \mathcal{A}$. 
A factor which is both left and right special is called \emph{bispecial}.
If $\uu$ is aperiodic, then for any length $n$ at least one factor $w \in \mathcal{L}_\uu(n)$ is left special and at least one factor $v\in \mathcal{L}_\uu(n)$ is right  special.    

Factors of an infinite word $\uu$ can be visualized by the so-called \emph{Rauzy graphs} $\Gamma_\uu(n)$, $n \in \mathbb{N}$. 
The set of vertices of $\Gamma_\uu(n)$ is $\mathcal{L}_\uu(n)$  and the set of its edges is $\mathcal{L}_\uu(n+1)$.
An oriented edge $e\in  \mathcal{L}_\uu(n+1)$ starts in $u\in \mathcal{L}_\uu(n)$ and ends in $v\in\mathcal{L}_\uu(n)$ if $u$ is a prefix of $e$ and $v$ is a suffix of $e$.
If $w \in \mathcal{L}_\uu(n)$, we denote 
\begin{align*}
N_+(w) &= \{v \in \mathcal{L}_\uu(n)\,:\, w  \text{ is a prefix and} \ v \ \text{is a suffix of an edge}\  e \in \mathcal{L}_\uu(n+1)\}\,, \\
N_-(w) &= \{v \in \mathcal{L}_\uu(n)\,:\, v  \text{ is a prefix and} \ w \ \text{is a suffix of an edge}\  e \in \mathcal{L}_\uu(n+1)\}\,.
\end{align*} 
Any factor  $v \in \mathcal{L}_\uu(n+m)$  with a prefix $u\in \mathcal{L}_\uu(n)$ corresponds to an oriented path of length $m$ in $\Gamma_\uu(n)$ starting with the vertex $u$.  

The \emph{occurrence} of the word $w$ in $\uu = u_0u_1u_2 \cdots$ is every index $i\in \mathbb{N}$ such that $w$ is a prefix of the word $u_iu_{i+1}u_{i+2}\cdots$. 
The factor of length $n$ which occurs at the position $i$ is denoted by $f_n(i)$.   
Hence, $f_n(i)=w$ if $w \in \mathcal{L}_\uu(n)$ and $w$ is a prefix of $u_iu_{i+1}u_{i+2}\cdots$.
An infinite word $\uu$ is said to be \emph{recurrent} if each of its factors has at least two occurrences in $\uu$.
If $i<j$ are two consecutive occurrences of $w$ in $\uu$, then the word $u_{i}u_{i+1}\cdots u_{j-1}$ is called the \emph{return word} to $w$ in $\uu$. 
If the set of all return words to $w$ in $\uu$ is finite for each factor $w$ of $\uu$, the word $\uu$ is called \emph{uniformly recurrent}. 

A \emph{morphism} of the free monoid $\mathcal{A}^*$ is a map $\psi:\mathcal{A}^*\to \mathcal{A}^*$ such that $\psi(vw)=\psi(v)\psi(w)$ for all $v,w\in \mathcal{A}^*$.  
The \emph{incidence matrix} of $\psi$ is
$d\times d$ matrix $\boldsymbol{M}_{\!\psi}$ given by $[\boldsymbol{M}_{\!\psi}]_{ab}=|\psi(b)|_{a}$. 
The incidence matrix of $\psi$ can be used to compute the Parikh vector of the image of a word $w$ under $\psi$:
\begin{equation}\label{eq:parikh_vector_of_image_via_incid_matrix}
\vec{V}(\psi(w)) = \boldsymbol{M}_{\!\psi}\cdot\vec{V}(w)\,.
\end{equation}

The domain of a morphism $\psi$ of $\mathcal{A}^*$ can be naturally extended to $\mathcal{A}^{\N}$ by putting $\psi(\uu)=\psi(u_0u_1u_2\cdots ) = \psi(u_0)\psi(u_1)\psi(u_2)\cdots$. 
An infinite word $\uu$ is called a \emph{fixed point} of the morphism $\psi$ if $\uu=\psi(\uu)$.

\section{Arnoux--Rauzy words} \label{sec:arnoux_rauzy}

The Sturmian words can be described by many equivalent properties, for their list (which is far from being complete) see for example \cite{BaPeSt10}. 
These properties offer several possibilities for generalization. 
One of them was used by Arnoux and Rauzy in \cite{ArRo91} to introduce the words today known under their names. 

\begin{definition} \label{def:arnoux_rauzy}
A recurrent infinite word $\uu \in \mathcal{A}^\mathbb{N}$ is a \emph{$d$-ary Arnoux--Rauzy word} if for all $n$ it has $(d-1)n+1$ factors of length $n$ with exactly one left and one right special factor of length $n$. 
\end{definition}

Over the binary alphabet the Arnoux--Rauzy words coincide with the Sturmian words.  
The Arnoux--Rauzy words belong to a broader family of episturmian words (e.g., see \cite{GlJu09}).
They are also embedded in the very general concept of tree sets introduced in \cite{7people15} which comprises several generalizations of Sturmian words to multi-letter alphabet. 
The Arnoux--Rauzy words share many properties with the Sturmian words (e.g., see \cite{RiZa00, JuVu, Fo02}).
Here we recall some of them.
If $\uu$ is a $d$-ary Arnoux--Rauzy word, then 
\begin{itemize}
\item  there exists a dominant letter $a \in \mathcal{A}$ such that  $a$ occurs in each factor from $\mathcal{L}_\uu(2)$;
\item  $ \mathcal{L}_\uu$ is closed under reversal, i.e.,  $w=w_0w_1\cdots w_{n-1} \in \mathcal{L}_\uu$  implies $\bar{w}=w_{n-1}\cdots w_1w_0 \in \mathcal{L}_\uu$;
\item each  bispecial factor $w$ of $\uu$ is a  palindrome, i.e.,  $w=\bar{w}$; 
\item $\uu $ is uniformly recurrent;  
\item any factor of $\uu$ has exactly $d$ return words in $\uu$. 
\end{itemize}

On the other hand, some properties of Sturmian words are not present in $d$-ary Arnoux--Rauzy  words when $d\geq 3$. 
An example of such a property is  the so-called balancedness.
Already Hedlund and Morse \cite{MoHe40} proved that a binary aperiodic word $\uu$  is Sturmian if and only if for any pair $v, w \in \mathcal{L}_\uu$ of factors of the same length the inequality 
$|v|_a - |w|_a \leq  c=1$ holds for any letter $a \in \mathcal{A}$.
This property is not preserved in Arnoux--Rauzy words, even if the constant $c=1$ is allowed to depend on $d$. 
For a detailed study of this problem, see \cite{BeCaSt13}. 
  
If each prefix of an Arnoux--Rauzy word $\uu$ is left special, then $\uu$ is called \textit{standard}.
For each Arnoux--Rauzy word ${\bf v}$, there exists a unique standard Arnoux--Rauzy word $\uu$ such that $\mathcal{L}_\uu = \mathcal{L}_\vv$. 
We will work with the $S$-adic representation of the Arnoux--Rauzy words as described in \cite{Fo02}.
Therefore we define the set $S$ of elementary morphisms over the alphabet $\mathcal{A} = \{0,1,\ldots,d-1\}$.
\begin{equation} \label{eq:elementary_morphisms}
\text{ For } i=0,1,\ldots, d-1 \quad  \text{ we put } \quad    \varphi_i  \ : \  
\left\{ \begin{array}{l}
i\to i\,;\\
j\to ij \ \ \text{for } j\neq i\,. 
\end{array} \right.
\end{equation}

Any standard Arnoux--Rauzy word $\uu$ is an image of a standard Arnoux--Rauzy word $\uu'$ under a morphism $\varphi_i$, where the letter $i$ coincides with the dominant letter of $\uu$.  
This property enables us to assign to any standard Arnoux--Rauzy word  a sequence  $(i_n)_{n\geq 0}$  of indices and a sequence $\bigl(\uu^{(n)}\bigr)_{n\geq 0}$ of standard Arnoux--Rauzy words such that  
\begin{equation} \label{eq:directive_sequence} 
\uu=\uu^{(0)}  \  \text{ and }  \ \uu^{(n)} = \varphi_{i_n}\bigl(\uu^{(n+1)}\bigr) \ \text{ for each } n \in \mathbb{N}. 
\end{equation}
The sequence  $(i_n)_{n\geq 0}$ is called the \textit{directive sequence} of $\uu$.  
 
\medskip

For any standard Arnoux--Rauzy word $\uu$, both sequences $(i_n)_{n\geq 0}$ and $\bigl(\uu^{(n)}\bigr)_{n\geq 0}$ are uniquely given.
Moreover, every letter $i \in \mathcal{A}$ occurs in $(i_n)_{n\geq 0}$ infinitely many times.
On the other hand, a sequence $(i_n)_{n\geq 0}$ which contains each letter of $\mathcal{A}$ infinitely many times determines a unique Arnoux--Rauzy word and thus the unique sequence $\bigl(\uu^{(n)}\bigr)_{n\geq 0}$, cf. \cite{RiZa00}.   
  
\begin{example} \label{exam:d_bonacci}
The most famous Sturmian word is the \emph{Fibonacci word} which is the fixed point of so-called \emph{Fibonacci morphism} defined as $\tau: 0 \to 01$, $1 \to 0$.
Analogously, for every integer $d \geq 2$ we define the \emph{$d$-bonacci word} $\tt$ as the fixed point of the \emph{$d$-bonacci morphism} 
\begin{equation*} \label{eq:d_bonacci}
\tau  \ : \  
\left\{ 
\begin{array}{ccl}
a &\to& 0(a+1) \quad \text{ for } a = 0, \ldots, d-2\,,\\   
(d-1)& \to &0\,. 
\end{array}    
\right. 
\end{equation*}
It is a $d$-ary standard Arnoux--Rauzy word. 
By simple computations we get $\tau^d = \varphi_0\varphi_1\cdots \varphi_{d-1}$ and so its directive sequence $(i_n)_{n \geq 0}$ is $(0\, 1\, 2 \cdots d-1)^\omega$, i.e., its $n^{th}$ element $i_n \in \mathcal{A}$ satisfies $i_n \equiv n \mod d$ for any $n \in \N$.
Over a ternary alphabet the word and the corresponding morphism is usually called \emph{Tribonacci} word and morphism, respectively.
\end{example}

\section{Special factors and non-repetitive complexity} \label{sec:special_factors_and_nonrepetitive_complexity}
 
First we show the role that special factors play in the evaluation of non-repetitive complexity. Let us recall that  for a given infinite word $\uu$ we denoted by $f_n(i)$ the factor of length $n$  occurring in $\uu$ at the position $i$. 
\begin{lemma} \label{lem:nonrepetitive_complexity} 
Let $\uu = u_0u_1u_2\cdots$ be a recurrent aperiodic infinite word, $n \in \mathbb{N}$ and $m=nr\mathcal{C}_\uu(n)$. 
Then there exists $h \in \mathbb{N}$ such that 
\begin{itemize}
\item the set $L=\{f_n(h), f_n(h+1), \ldots, f_n(h+m-1)\} $  contains $m$ distinct factors  of $ \mathcal{L}_\uu(n)$;
\item the factor $f_n(h-1)$ is right special and belongs to $L$;
\item the factor $f_n(h+m)$ is left special and belongs to $L$. 
\end{itemize}
\end{lemma}

\begin{proof} 
Let $k$ be an integer such that the factors from $L'=\{f_n(k), f_n(k+1), \ldots, f_n(k+m-1)\} $ are pairwise distinct.
As $\uu$ is recurrent, we can assume $k\geq 1$.
Since $m$ is the maximal number of distinct consecutive factors, there exist integers $i$ and $j$ such that 
$$
k\leq i, j\leq k+m-1, \quad f_n(k-1) = f_n(i) \quad \text{ and }  \quad f_n(k+m) = f_n(j)\,.
$$ 

We discuss  two cases.   

\noindent {\bf  Case I:} Assume $k<j$  and  $i< k+m-1$.
 As $f_n(k+m) = f_n(j)$, the factors $f_{n+1}(k+m-1)$ and  $f_{n+1}(j-1)$ of length $n+1$ have a common suffix of length $n$. 
It follows that $u_{k+m-1}\neq  u_{j-1}$.  
Otherwise $f_n(k+m-1)$ and $f_n(j-1)$ would coincide, which is a contradiction with our choice of $k$.  
It means that $ u_{k+m-1}f_n(j)$ and $ u_{j-1}f_n(j)$ both belong to the language $\mathcal{L}_\uu$. Thus $f_n(k+m) = f_n(j)$ is a left special factor. 
Analogously one can show that $f_n(k-1)$ is a right special factor.
Thus we can choose $h=k$ and $L=L'$. 

\medskip

\noindent {\bf  Case II:} Assume $k=j$  or  $i=k+m-1$.
Without loss of generality we may assume $k=j$, i.e., $f_n(k) = f_n(k+m)$.  
Aperiodicity of $\uu$ guarantees that there exists $\ell \in \mathbb{N}$ such that 
\begin{equation*} 
f_n(k+q ) = f_n(k+m+q)\  \text{ for each } q=0,1,\ldots,\ell \quad  \text{ and }  \quad  f_n(k+\ell+1 ) \neq  f_n(k+m+\ell+1). 
\end{equation*}
Therefore $\{ f_n(k+\ell+1), f_n(k+\ell+2), \ldots, f_n(k+\ell+m) \}=L'$. 
We set $h=k+\ell+1$ and we show that the factor $f_n(h-1)= f_n(k+\ell)$ is right special and the factor $f_n(h+m) = f_n(k+\ell+m+1)$ is left special.

Since $f_n(k+\ell)=f_n(k+\ell +m)$  and $f_n(k+\ell+1)\ne f_n(k+\ell +m+1)$, the letters $u_{k+\ell+n}$ and $u_{k+\ell+m+n}$ differ. Hence, the factor $ f_n(k+\ell)$ is right special.
Since $m$ is the maximal number of distinct consecutive factors, $f_n(k+\ell +m+1) \in L'$.
By definition of $\ell$, $f_n(k+\ell +m+1)\ne  f_n(k+\ell+1)$, and so $f_n(k+\ell +m+1) = f_n(k+\ell + p)$ for some $1 < p \leq m$.
We conclude that $ f_n(k+\ell +m+1)$ is left special using the same arguments as in Case I.
\end{proof}

\begin{theorem} \label{thm:nonrepetitive_complexity_sturmian} 
Let $\uu$ be a Sturmian word.  
Then $nr\mathcal{C}_\uu(n) =n+1$ for every $n \in \mathbb{N}$. 
\end{theorem}

\begin{proof}  
Let $n\in\N$.
Any Sturmian word $\uu$ has exactly one left and one right special factor of length $n$. 
Let us denote them $\alpha$ and $\beta$, respectively.
Therefore, in the Rauzy graph $\Gamma_\uu(n)$ the vertex $\alpha$ has indegree $2$ and all other vertices have indegree $1$ and the vertex $\beta$ has outdegree $2$ and all others vertices have outdegree $1$. 
Thus $\Gamma_\uu(n)$ is a union of two cycles $C_0$ and $C_1$ which have a common part, namely the path from $\alpha$ to $\beta$ (see Figure \ref{fig:sturmian_rauzy_graph}). 
Denote by $\gamma, \delta,\zeta,\eta$ the vertices such that $(\beta, \gamma), (\beta,\delta), (\zeta,\alpha)$ and $(\eta,\alpha)$  are edges in $\Gamma_\uu(n)$.

\begin{figure}[h]
\begin{center}
\includegraphics[scale=0.8]{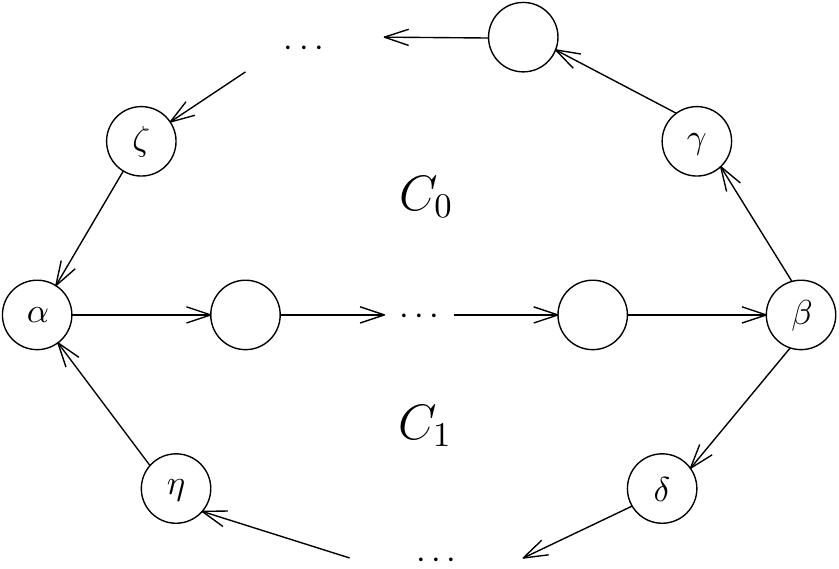}
\end{center}
\caption{The Rauzy graph $\Gamma_\uu(n)$ of the Sturmian word $\uu$.} \label{fig:sturmian_rauzy_graph}
\end{figure}

By Lemma \ref{lem:nonrepetitive_complexity}, let $h, m\in\N$ be such that  $m=\nrc_\uu(n)$, $L=\{f_n(h), f_n(h+1), \ldots, f_n(h+m-1)\}$, $\#L=\nrc_\uu(n)$, $\alpha=f_n(h+m)\in L$ and $\beta=f(h-1) \in L$. 
Hence in the Rauzy graph $\Gamma_\uu(n)$, there exists a path starting in a vertex of $N_+(\beta)$, passing through $\alpha$ and $\beta$, and ending in a vertex of $N_-(\alpha)$. Moreover, this path cannot pass twice through the same vertex. 
The only possible paths are the 
$$\gamma \to \cdots \to \alpha \to \cdots \to \beta \to \delta\to\cdots\to \eta \text{ and }
\delta \to \cdots \to \alpha \to \cdots \to \beta \to \gamma\to\cdots\to \zeta.$$
Both paths are hamiltonian, i.e., they are passing through all vertices of $\Gamma_\uu(n)$ exactly once.
It follows that $\nrc_\uu(n)=\mathcal{C}_\uu(n)=n+1$.
\end{proof}

The previous theorem  states  that the  factor complexity and the  non-repetitive complexity coincide for Sturmian words.  
In the next section we prove that this property is not preserved in $d$-ary Arnoux--Rauzy words with $d\geq 3$.   
Nevertheless, the equality $nr\mathcal{C}_\uu=\mathcal{C}_\uu$ we observed in binary aperiodic words with the smallest factor complexity can take place also in a word with the maximal factor complexity, as shown in \cite{NiRa16}.
The next corollary of Lemma \ref{lem:nonrepetitive_complexity} illustrates that the equality $nr\mathcal{C}_\uu=\mathcal{C}_\uu$ forces the Rauzy graphs of a word $\uu$ to have a very special form. 

\begin{corollary} \label{coro:property_of_rauzy_graph}
Let $\uu$ be a recurrent aperiodic word, $n \in \mathbb{N}$, $w \in  \mathcal{L}_\uu(n)$ and $m=nr\mathcal{C}_\uu(n)$.
Let $h\in\N$ be such that $f_n(h-1)$ and $f_n(h+m)$ are respectively the right and left  special factors from Lemma \ref{lem:nonrepetitive_complexity}. Assume $\nrc_\uu(n)=\mathcal{C}_\uu(n)$.
\begin{enumerate}
\item If $w \neq f_n(h-1)$, then $N_+(w)$ contains at least $\#N_+(w)-1$ left special factors. 
\item If $w \neq f_n(h+m)$, then $N_-(w)$ contains at least $\#N_-(w)-1$ right special factors. 
\end{enumerate}
\end{corollary}

\begin{proof}
If the factor $w$ is not right special, then the set $ N_+(w)$  consists of one element and the statement is trivial. 
Let $w \neq f_n(h-1)$  be a right special factor.
We write it in the form $w=as$, where $a \in \mathcal{A}$ and $s \in \mathcal{L}_\uu(n-1)$.  
We denote $q=\#N_+(w)$ and find distinct letters $b_1, b_2, \ldots, b_q$ such that $N_+(w) = \{sb_1,sb_2, \ldots, sb_q\}$.
Obviously, $asb_k \in \mathcal{L}_\uu(n+1) $ for each $k=1,2,\ldots,q$.  
The assumption $nr\mathcal{C}_\uu(n) =\mathcal{C}_\uu(n) $ implies that $sb_k$ occurs in the set $L$ described in Lemma \ref{lem:nonrepetitive_complexity}. 
It means that $f_n(h+j_k) = sb_k $ for some index $j_k$, $0\leq j_k \leq m-1$. Moreover, there exists an index $p$, $0\leq p\leq m-1$ such that $f_n(h+p) = w=as$.   
 
Let us look at the letter which precedes $f_n(h+j_k) = sb_k$, i.e., at the letter $u_{h+j_k-1}$: 
\begin{itemize}
\item[--] if $h= h+ j_k$, then $u_{h+j_k-1}\neq a$ as $w=as \neq f_n(h-1)$;
\item[--] if $h\neq h+j_k \neq h+p+1$, then $u_{h+j_k-1}\neq a$, otherwise the factor $as = f_n(h+p) = f_n(h+j_k-1)$ occurs twice in $L$, which is a contradiction.
\end{itemize}
We showed that for all $k =1,2,\ldots, q$  (up to one possible exception  when  $h+j_k =  h+p+1$), the factors $asb_k$ and $u_{h+j_k-1} sb_k$ belong to the language of $\uu$ and $u_{h+j_k-1} \neq a$.
It means that $sb_k$ is a left special factor. 
 
The proof of the second part of the statement is analogous. 
\end{proof} 

\section{Non-repetitive complexity of Arnoux--Rauzy words} \label{sec:nonrepetitive_complexity_in_arnoux_rauzy_words}
 
For every Arnoux--Rauzy word $\uu$, there exists at most one bispecial factor of $\uu$ of length $n$.  
Thus we can order the bispecial factors by their lengths:
for $k \in \N$ we denote $B_\uu(k)$ the $k^{th}$ bispecial factor of $\uu$.
In particular, $B_\uu(0) = \varepsilon$, $B_\uu(1) = u_0$ (the first letter of $\uu$), etc. 

Now we can formulate the link between the lengths of the return words to the bispecial factors and the values of non-repetitive complexity. 
Let us recall that any factor of a $d$-ary Arnoux--Rauzy word $\uu$ has exactly $d$ return words, cf. \cite{JuVu}.

\begin{proposition} \label{prop:arnoux_rauzy_nonrepetitive_complexity}  
Let $\uu$ be a $d$-ary Arnoux--Rauzy word and let $n, k \in \mathbb{N}$ be such that $B_\uu(k-1) < n \leq B_\uu(k)$.
Denote by $r_0,r_1,\ldots, r_{d-1}$ the return words to  $B_\uu(k)$ in $\uu$.
\begin{enumerate}
\item If $n =|B_\uu(k)|$, then 
$$
nr\mathcal{C}_\uu(n) = \max\{|r_ir_j|\,  :\, r_ir_j \in \mathcal{L}_\uu, \ 0\leq i,j \leq d-1, i \neq j\} -1\,.
$$
\item If $|B_\uu(k-1)| < n < | B_\uu(k)|$, then
$$
nr\mathcal{C}_\uu(n) =  nr\mathcal{C}_\uu\bigl(|B_\uu(k)|\bigr) - |B_\uu(k)| + n\,.
$$
\end{enumerate}
\end{proposition}  
  
\begin{proof}
Let $\uu$ be a $d$-ary Arnoux--Rauzy word. 
Its $n^{th}$ Rauzy graph $\Gamma_\uu(n)$ contains exactly one vertex $\alpha$ with the indegree $d$ and all other vertices have indegree $1$. 
It also contains exactly one vertex $\beta$ with the outdegree $d$, all other vertices have outdegree $1$. 
It means that $\Gamma_\uu(n)$ is composed of $d$ cycles $C_0, C_1, \ldots, C_{d-1}$ which only have in common the path from $\alpha$ to $\beta$ (see Figure \ref{fig:arnoux_rauzy_rauzy_graph}). 
For every $i \in \{0, \ldots, d-1\}$, we denote by $\ell_i$ the number of vertices in the cycle $C_i$ and by $\gamma_i, \zeta_i$ the vertices from the cycle $C_i$ such that $(\beta, \gamma_i), (\zeta_i, \alpha)$ are edges in $\Gamma_\uu(n)$. 
Let $p$ be the number of vertices on the minimal path from $\alpha$ to $\beta$.

Let $h\in\N$ be such that $L =\{ f_n(h), f_n(h+1),\ldots, f_n(h+m-1)\}$ is the set from Lemma \ref{lem:nonrepetitive_complexity} with $m= \# L = \nrc_\uu(n)$. 
Then $f_n(h-1) = \beta $, $f_n(h+m)=\alpha$ and $\alpha,\beta\in L$. 
Hence the path in $\Gamma_\uu(n)$ corresponding to $L$ is of the form:
$$\gamma_i\to \cdots\to \zeta_i \to \alpha \to \cdots \to  \beta \to \gamma_j\to \cdots \to \zeta_j$$
for some $i,j\in\mathcal{A}, i \neq j,$ and it contains $\nrc_\uu(n) = \ell_i + \ell_j -p $ vertices. So it suffices to compute the numbers $\ell_i$, $\ell_j$ and $p$. 

\begin{figure}[h]
\begin{center}
\includegraphics[scale=0.8]{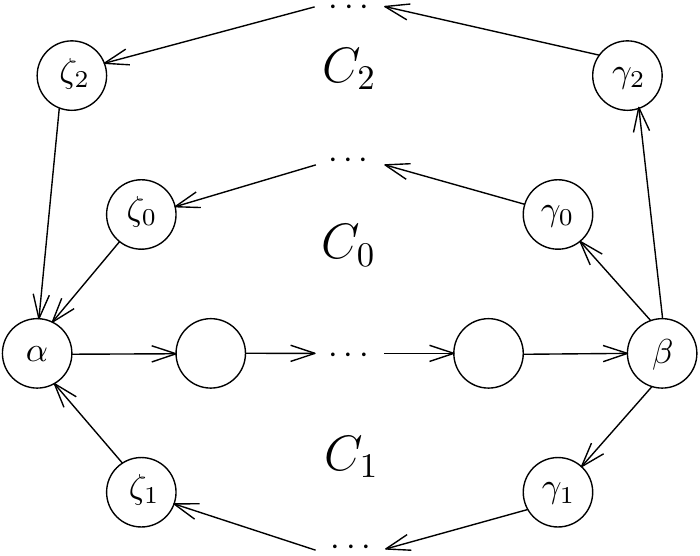}
\end{center}
\caption{The Rauzy graph $\Gamma_\uu(n)$ of a ternary Arnoux--Rauzy word $\uu$.} \label{fig:arnoux_rauzy_rauzy_graph}
\end{figure}

\medskip

$(1)$: 
If $n = |B_\uu(k)|$, then $\alpha = \beta = B_\uu(k)$, $p=1$ and the Rauzy graph $\Gamma_\uu(n)$ contains $d$ cycles $C_0, C_1, \ldots, C_{d-1}$ with only the vertex $B_\uu(k)$ in common.
Clearly, these cycles correspond with the return words to $B_\uu(k)$: if we start in $B_\uu(k)$ and concatenate the first letters of all vertices of $C_i$, we get $r_i$ for all $i \in \mathcal{A}$. 
Thus the number $\ell_i$ of vertices in $C_i$ is equal to $|r_i|$. 
Hence $nr\mathcal{C}_\uu(n) = \ell_i + \ell_j - 1 = |r_{i}| + |r_{j}| - 1$ for some $i \neq j$.
By definition of $\nrc_\uu(n)$, we have to choose $i \neq j$ such that the word $r_{i}r_{j}$ is a factor of $\uu$ and its length is maximal possible.  

\medskip

$(2)$: 
If $|B_\uu(k-1)| < n < |B_\uu(k)|$, then $\alpha \neq \beta$ and $p > 1$. 
Observe that if $p > 1$, then $\Gamma_\uu(n+1)$ has also cycles of lengths $\ell_i$ for all $i \in \mathcal{A}$ and the minimal path from the left special factor to the right special factor contains $p -1$ vertices.
It follows that $\Gamma_\uu(n+p-1)$ contains the bispecial factor $B_\uu(k)$ and so $|B_\uu(k)| = n+p-1$. 
By the Rauzy graph $\Gamma_\uu(|B_\uu(k)|)$ we have $nr\mathcal{C}_\uu(|B_\uu(k)|) = \ell_{i} + \ell_{j} - 1$, as the lengths of the cycles are preserved. So  
$$\nrc_\uu(n) = \ell_i + \ell_j -p = nr\mathcal{C}_\uu(|B_\uu(k)|) +1 - (|B_\uu(k)|-n +1) = nr\mathcal{C}_\uu(|B_\uu(k)|)- |B_\uu(k)| +n\,.$$
\end{proof}

In the introduction we stated the inequality between the   recurrence function $R_{\uu}$  and the non-repetitive complexity.
It is worth mentioning that $R_{\uu}$ is also linked to return words, as stated by Cassaigne in \cite{Ca99}.

\begin{proposition}[\cite{Ca99}] \label{prop:reccurrence_function_by_return_words} 
Let $\uu$ be a recurrent infinite word. 
Then for each $n \in \mathbb{N}$, 
$$
R_{\uu}(n) - n+1 = \max\{ |r|\, : \,  r \ \text{is a return word to } w \in \mathcal{L}_\uu(n)\}.  
$$
\end{proposition}

\medskip

To transform Proposition \ref{prop:arnoux_rauzy_nonrepetitive_complexity} into an explicit formula for $nr\mathcal{C}_\uu$, we have to compute the lengths of the return words to the bispecial factors in $\uu$ and also decide which return words are neighbouring in $\uu$. 
For this purpose we will essentially use the directive sequence $(i_n)_{n \geq 0}$ of a standard Arnoux-Rauzy word $\uu$ introduced in Section \ref{sec:arnoux_rauzy}. 
Let us emphasize that the non-repetitive complexity of $\uu$ depends only on the language $\mathcal{L}_\uu$ and not on the word $\uu$ itself. Since for every Arnoux-Rauzy word $\uu$ there exists a unique standard Arnoux-Rauzy word $\vv$ such that $\mathcal{L}_\uu = \mathcal{L}_\vv$, we can restrict our considerations only to standard Arnoux--Rauzy words. 
Note that if $\uu$ is standard Arnoux--Rauzy word, all its bispecial factors are prefixes of $\uu$.

The following notion of \textit{derived word} which codes the order of the return words in $\uu$ will be also useful.  

\begin{definition}  \label{defi:derived_word}
Let $w$ be a prefix of a uniformly recurrent word $\uu$ and let $r_0,r_1, \ldots,r_{\ell-1}$ be the return words to $w$ in $\uu$. 
If we write $\uu$ as a concatenation $\uu = r_{j_0}r_{j_1}r_{j_2}\cdots$, then the word $ j_0j_1j_2\cdots$  is called the \emph{derived word} to $w$ in $\uu$ and is denoted $\Der_\uu(w)$. 
\end{definition}
 
We do not specify the order of the return words and thus the derived word is determined uniquely up to a permutation of letters.
Clearly, the derived word to the empty word $\varepsilon$ in $\uu$ is the word $\uu$ itself. 
The simple form of the morphisms $\varphi_i$ defined by \eqref{eq:elementary_morphisms} gives immediately the following claim, which can be also deduced from the results in \cite{JuVu} or \cite{Med}.

\begin{claim} \label{claim:derived_word_of_preimage}
Let $\uu$ and ${\bf v}$ be standard $d$-ary  Arnoux--Rauzy words such that ${\bf v} = \varphi_i(\uu)$ with $i \in \mathcal{A}$.
Then for any $k\in \mathbb{N}$ it holds:  
\begin{itemize}  
\item $ B_{\bf v}(k+1) = \varphi_i\bigl(B_\uu(k)\bigr)i $;
\item if $r_0,r_1,\ldots, r_{d-1}$ are  the return words to $B_\uu(k)$ in $\uu$, then $ \varphi_i(r_0),\varphi_i(r_1), \ldots, \varphi_i(r_{d-1})$ are the return words to $ B_{\bf v}(k+1)$ in ${\bf v}$;
\item $\Der_\uu\bigl(B_\uu(k)\bigr) = \Der_{\bf v}\bigl(B_{\bf v}(k+1)\bigr)$ up to permutation of letters.
\end{itemize}
\end{claim}

\begin{corollary} \label{coro:history_of_bispecial_factors} 
Let $\uu$ be a standard Arnoux--Rauzy word with the directive sequence $(i_n)_{n\geq 0}$ and $\bigl(\uu^{(n)}\bigr)_{n\geq 0}$ be the sequence satisfying  \eqref{eq:directive_sequence}.
Then the derived word to $B_\uu(k)$ in $\uu$ is (up to permutation of letters) the word $\uu^{(k)}$ and the corresponding return words are $\psi(0)$, $\psi(1), \ldots$, $\psi(d-1)$, where $\psi = \varphi_{i_0}\varphi_{i_1}\cdots \varphi_{i_{k-1}}$.  
\end{corollary}

\begin{proof}
Obviously, the bispecial factor $\varepsilon$ has in the word $\uu^{(k)}$ the return words $0, 1, \ldots, d-1$ and the derived word (up to permutation of letters) to $B_{\uu^{(k)}}(0) =\varepsilon$  in $\uu^{(k)}$ is $\uu^{(k)}$. 
By repeated application of Claim \ref{claim:derived_word_of_preimage} we get 
$$
\Der_{\uu}(B_\uu(k)) = \Der_{\uu^{(1)}}(B_{\uu^{(1)}}(k-1)) = \cdots = \Der_{\uu^{(k)}}(B_{\uu^{(k)}}(0)) = \uu^{(k)}
$$
and 
$$
\{r_0, \ldots, r_{d-1}\} = \{\varphi_{i_0}\varphi_{i_1}\cdots \varphi_{i_{k-1}}(0), \ldots, \varphi_{i_0}\varphi_{i_1}\cdots \varphi_{i_{k-1}}(d-1)\}\,.
$$
\end{proof}

Corollary \ref{coro:history_of_bispecial_factors} enable us to express the return words to the $k^{th}$ bispecial factor $B_\uu(k)$. 
However, we also need to know which return words are neighbouring, i.e., for which $i \neq j$ the word $r_ir_j$ is a factor of $\uu$. 
Corollary \ref{coro:history_of_bispecial_factors} transforms this question to the description of neighbouring letters in the Arnoux--Rauzy word $\uu^{(k)}$ with the directive sequence $(i_{n+k})_{n\geq 0}$, which is trivial. 

\begin{claim}  \label{claim:neighbouring_letters}
Let $\uu$ be a standard Arnoux--Rauzy word with the directive sequence $(i_n)_{n\geq 0}$. 
Then $i_0$ is the dominant letter in $\uu$ and the factors of length $2$ in $\uu$ are the words $i_0a, ai_0$ for all $a \in \mathcal{A}$. 
\end{claim}

For every $k \in \N$ and every letter $a \in \mathcal{A}$ we define  $S_a(k) = \sup \{\ell : 0 \leq \ell < k, i_\ell = a\}$. 
As usual, if the set is empty, i.e., $i_\ell \neq a$ for all $\ell < k$, then  $S_a(k) = -\infty$. 
Let us emphasize that $S_a(k)  =S_b(k)$ for two distinct  letters $a$ and $b$ if and only if $S_a(k)  =S_b(k) = -\infty$.

\begin{theorem} \label{thm:nonrepetitive_complexity_arnoux_rauzy}
Let $\uu$ be a $d$-ary Arnoux--Rauzy word. 
For every integer $n\geq 1$ we take the unique $k$ such that $|B_\uu(k-1)| < n \leq |B_\uu(k)|$.
Then we have
$$
nr\mathcal{C}_\uu(n) = |\varphi_{i_0}\varphi_{i_1} \cdots \varphi_{i_{k-1}}\varphi_{i_{k}}(a)| -1 - |B_\uu(k)| + n\,,$$
where $(i_n)_{n \geq 0}$ is the directive sequence of the standard Arnoux-Rauzy word with the language $\mathcal{L}_\uu$ and $a \in \mathcal{A}$ is any letter different from $i_k$ such that  
$ S_a(k) = \inf \{S_b(k) : b  \in \mathcal{A}, b \neq i_k \}$.
\end{theorem}

\begin{proof}  
Since the function $nr\mathcal{C}_\uu$ depends only on the language $\mathcal{L}_\uu$ and not on the word $\uu$ itself, we can work with the standard Arnoux-Rauzy word $\vv$ such that $\mathcal{L}_\vv = \mathcal{L}_\uu$ instead of $\uu$.
We denote  $(i_n)_{n \geq 0}$ the directive sequence of $\vv$ and to simplify the notation we also denote $\psi = \varphi_{i_0}\varphi_{i_1} \cdots \varphi_{i_{k-1}}$.

We start from Proposition \ref{prop:arnoux_rauzy_nonrepetitive_complexity} and using the previous claims we express $nr\mathcal{C}_\vv(|B_\vv(k)|)$ more explicitly. 
By Corollary \ref{coro:history_of_bispecial_factors}  and Claim \ref{claim:neighbouring_letters} the admissible pairs of return words to $B_\vv(k)$ are 
$$
\{|r_ir_j|\,  :\, r_ir_j \in \mathcal{L}_\vv, \, 0\leq i,j \leq d-1, i \neq j\} = \{|\psi(i_ka)| : a \in \mathcal{A}, a \neq i_k\}\,.
$$
It suffices to determine for which letter $a \neq i_k$ the image $|\psi(a)|$ is the longest possible. 
Let us emphasize that for every $i \in \{0, \ldots, d-1\}$ and every words $x,y \in \mathcal{A}^*$ we have
\begin{itemize}
\item[(i)] $\varphi_i(xa) = \varphi_i(x)ia$ if $i \neq a$ and $\varphi_i(xa) = \varphi_i(x)i$ if $i = a$;
\item[(ii)] if $x$ is a proper prefix of $y$, i.e., $y = xz$ for some non-empty word $z$, then $\varphi_i(x)$ is a proper prefix of $\varphi_i(y) = \varphi_i(x)\varphi_i(z)$.
\end{itemize}  
For  two distinct letters $a, b \in \mathcal{A}$ we discuss two cases. 

-- If  $S_a(k) = S_b(k)$, then  the morphisms $\varphi_a, \varphi_b$ are not included in the decomposition of $\psi$. 
Thus by application of Item (i) we get $\psi(a) = x'a$ and $\psi(b) = x'b$ for some non-empty word $x' \in \mathcal{A}^*$ and so $|\psi(a)| = |\psi(b)|$.

-- If  $S_a(k)< S_b(k)$, then we split $\psi = \sigma\varphi_b\theta$ such that the decomposition of the morphism $\theta$ contains neither $\varphi_a$ nor $\varphi_b$. 
Then by Item (i) we have $\theta(a) = x'a$ and $\theta(b) = x'b$ for some non-empty word $x' \in \mathcal{A}^*$ and since $\varphi_b(x'a) = \varphi_b(x')ba$ and $\varphi_b(x'b) = \varphi_b(x')b$, the word $\varphi_b(\theta(b))$ is a proper prefix of $\varphi_b(\theta(a))$. 
By Item (ii) it means that also $\sigma(\varphi_b(\theta(b))) = \psi(b)$ is a proper prefix of $\sigma(\varphi_b(\theta(a))) = \psi(a)$ and so $|\psi(b)| < |\psi(a)|$. 

We may conclude that   
$$
nr\mathcal{C}_\vv(|B_\vv(k)|)  = \max\{|r_ir_j|\,  :\, r_ir_j \in \mathcal{L}_\vv, \, 0\leq i, j \leq d-1, i \neq j\} - 1= |\psi(i_ka)| -1 = |\psi\varphi_{i_{k}}(a)| -1\, ,
$$
where $a$ is any letter different from $i_k$ such that $S_a(k) = \inf \{S_b(k) : b \neq i_k \}$.
By Proposition \ref{prop:arnoux_rauzy_nonrepetitive_complexity} it concludes the proof, since for all $n, k \in \N$ we clearly have $B_\uu(k) = B_\vv(k)$ and $nr\mathcal{C}_\uu(n) = nr\mathcal{C}_\vv(n)$.
\end{proof}

\section{Initial non-repetitive complexity of standard Arnoux--Rauzy words}

The following lemma uses again the notation
$f_n(i)$ for the factor of length $n$  occurring in $\uu$ at the position $i$. 

\begin{lemma} \label{lem:initial_nonrepetitive_complexity}
Let $\uu = u_0u_1u_2\cdots$ be a recurrent infinite word, $n \in \mathbb{N}$ and   $m=inr\mathcal{C}_\uu(n)$. 
Then the set $L = \{f_n(0), f_n(1), \ldots, f_n(m-1)\}$ contains $m$ distinct factors of $\mathcal{L}_\uu(n)$ and the factor $f_n(m)$ is either left special and $f_n(m) = f_n(i)$ for some $i, 0<i<m$, or $f_n(m) = f_n(0)$. 
\end{lemma}

\begin{proof} 
The proof of Case I of Lemma \ref{lem:nonrepetitive_complexity} immediately gives this statement.  
\end{proof}

\begin{theorem} \label{thm:initial_nonrepetitive_complexity_arnoux_rauzy}
Let $\uu$ be a standard $d$-ary Arnoux--Rauzy word with the directive sequence $(i_n)_{n \geq 0}$. 
For every integer $n\geq 1$ we take the unique $k$ such that $|B_\uu(k-1)| < n \leq |B_\uu(k)|$. Then we have
$$
inr\mathcal{C}_\uu(n) =  |\varphi_{i_0}\varphi_{i_1} \cdots \varphi_{i_{k-1}}(i_k)|\,.
$$  
\end{theorem}

\begin{proof}
Let $\uu$ be a standard $d$-ary Arnoux--Rauzy word and $n \in \N$.
We denote $m = inr\mathcal{C}_\uu(n)$ and $L = \{f_n(0), \ldots, f_n(m-1)\}$ the set from Lemma \ref{lem:initial_nonrepetitive_complexity}. 
Then $f_n(m) = f_n(0)$, since the word $f_n(0)$ is the only left special factor of $\uu$ of length $n$.  
It means that $m$ is equal to the length of the first return word to $f_n(0)$. 
We now determine its length.

If $n = |B_\uu(k)|$ for some $k \in \N$, it means that $f_n(0) = B_\uu(k)$ is bispecial factor.
Then by Corollary \ref{coro:history_of_bispecial_factors} the first return word to $f_n(0)$ is equal to the word $\varphi_{i_0}\varphi_{i_1} \cdots \varphi_{i_{k-1}}(i_k)$, since the word $\uu^{(k)}$ is standard  and so it starts with its dominant letter, which is by Claim \ref{claim:neighbouring_letters} the letter $i_k$. 
Thus $m = |\varphi_{i_0}\varphi_{i_1} \cdots \varphi_{i_{k-1}}(i_k)|$.

If $|B_\uu(k-1)| < n < |B_\uu(k)|$, then $B_\uu(k) = f_n(0)w$ for some non-empty word $w \in \mathcal{A}^*$ since all prefixes of $\uu$ are left special factors. 
Moreover, the word $f_n(0)$ is in $\uu$ always followed by the word $w$.
Indeed, since $f_n(0)$ is not right special, there is a unique letter $a \in \mathcal{A}$ such that $f_n(0)a \in \mathcal{L}_\uu$ and we can repeat the same process until we reach $B_\uu(k)$. 
But it means that the words $f_n(0)$ and $B_\uu(k)$ have the same return words and derived words and so the first return word to $f_n(0)$ is equal to the word $\varphi_{i_0}\varphi_{i_1} \cdots \varphi_{i_{k-1}}(i_k)$. 
Thus $m = |\varphi_{i_0}\varphi_{i_1} \cdots \varphi_{i_{k-1}}(i_k)|$.
\end{proof}

Let us emphasize that for non-standard Arnoux--Rauzy words the evaluating of the initial non-repetitive complexity is much more complicated, as, unlike the standard case, we do not have the control over the positions of the vertices corresponding to prefixes in the respective Rauzy graphs.

\begin{corollary} \label{thm:initial_nonrepetitive_complexity_sturmian} 
Let $\uu$ be a standard Sturmian word.  
Then $inr\mathcal{C}_\uu(n) =n+1$ for infinitely many  $n \in \mathbb{N}$.
\end{corollary}

\begin{proof}  
Let $(i_\ell)_{\ell  \geq 0}$ denote the directive sequence of $\uu$.
We will prove that $inr\mathcal{C}_\uu(n) = n+1$ for every $n$ such that $n = |B_\uu(k)| +1$ for some $k \in \N$ and $i_k \neq i_{k+1}$.
Since the directive sequence $(i_\ell)_{\ell  \geq 0}$ contains both letters $0$ and $1$ infinitely many times, it implies the statement of the corollary.

We take $n = |B_\uu(k)| +1$ such that $i_k \neq i_{k+1}$ and denote $r_0$ the more frequent return word to $B_\uu(k)$ and $r_1$ the other return word. 
By Corollary \ref{coro:history_of_bispecial_factors} and Claim \ref{claim:neighbouring_letters} we have $r_0 = \varphi_{i_0}\varphi_{i_1} \cdots \varphi_{i_{k-1}}(i_k)$ and $r_1 = \varphi_{i_0}\varphi_{i_1} \cdots \varphi_{i_{k-1}}(i_{k+1})$.
It also implies that $r_1$ is always followed by $r_0$, while $r_0$ can be followed both by $r_0$ and $r_1$. 

As explained before, the Rauzy graph $\Gamma_\uu({n-1})$ is composed of two cycles $C_0$ and $C_1$ with only the vertex $B_\uu(k)$ in common (see Figure \ref{fig:sturmian_initial_rauzy_graph}).
Moreover, these cycles correspond with the return words $r_0$ and $r_1$: if we start in $B_\uu(k)$ and concatenate the first letters of vertices from $C_0$, we get the return word $r_0$.
Thus the number of vertices of $C_0$ is equal to $|r_0|$. 
It is analogous for $C_1$ and $r_1$. 
This connection also means that the cycle $C_1$ is always followed by $C_0$, while $C_0$ can be followed by both $C_0$ and $C_1$. 

We denote $\alpha$ the edge from the cycle $C_0$ outcoming from the vertex $B_\uu(k)$ and $\beta$ the edge from $C_0$ incoming to $B_\uu(k)$ (see Figure~\ref{fig:sturmian_initial_rauzy_graph}).
Then $\alpha$ is the left special factor of $\uu$  of length $n$ and $\beta$ is the right special factor of $\uu$ of length $n$.
It means that the Rauzy graph $\Gamma_\uu({n})$ is composed of the cycle with $|r_0|$ + $|r_1|$ vertices and one extra edge going from the vertex $\beta$ to the vertex $\alpha$ (see Figure~\ref{fig:sturmian_initial_rauzy_graph}). 
It follows that $|r_0| + |r_1| = n+1$. 
Moreover, the return words to the factor $\alpha$ are $r_0$ and $r_0r_1$ and $|r_0r_1| = |r_0| + |r_1| =  n+1$. 
Finally, it suffices to apply Theorem \ref{thm:initial_nonrepetitive_complexity_arnoux_rauzy} for $B_\uu(k) < n < B_\uu(k+1)$ such that $i_k \neq i_{k+1}$:
$$
inr\mathcal{C}_\uu(n) = |\varphi_{i_0}\varphi_{i_1} \cdots \varphi_{i_{k}}(i_{k+1})| = |\varphi_{i_0}\varphi_{i_1} \cdots \varphi_{i_{k-1}}(i_ki_{k+1})| = |r_0r_1| = n+1.
$$
\begin{figure}[h]
\begin{center}
\includegraphics[scale=0.85]{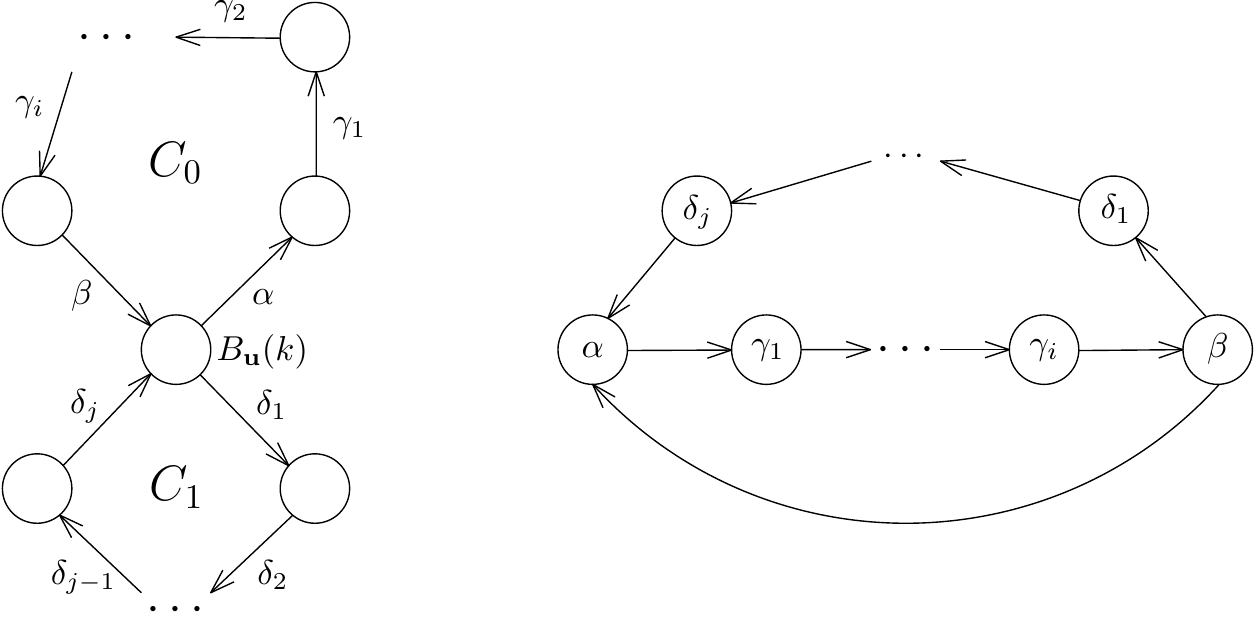}
\end{center}
\caption{The Rauzy graphs $\Gamma_\uu(n-1)$ (left) and $\Gamma_\uu(n)$ (right) of the standard Sturmian word $\uu$ for $n = |B_\uu(k)| +1$.} \label{fig:sturmian_initial_rauzy_graph}
\end{figure}
\end{proof} 

Recently, Bugead and Kim \cite{BuKi19} proved the new characterization of Sturmian words using the initial non-repetitive function: an infinite word $\uu$ is Sturmian if and only if $inr\mathcal{C}_\uu(n) \leq n+1$ for all $n \in \N$ with the equality for infinitely many $n$.  
So their result is more general than the previous corollary.

\section{Enumeration of non-repetitive complexity for $d$-bonacci word}

In this section we demonstrate the usefulness of Theorems \ref{thm:nonrepetitive_complexity_arnoux_rauzy} and \ref{thm:initial_nonrepetitive_complexity_arnoux_rauzy} on the $d$-bonacci words.
Let us recall that the $d$-bonacci word $\tt$ (see Example \ref{exam:d_bonacci}) is the fixed point of the morphism
\begin{equation*} 
\tau  \ : \  
\left\{ 
\begin{array}{ccl}
a &\to& 0(a+1) \ \text{ for } a = 0, \ldots, d-2\,\\   
(d-1)& \to &0\,
\end{array}    
\right. 
\text{ with the matrix } 
\boldsymbol{M}= 
\left( \begin{matrix}
1 & 1 & 1 & \cdots & 1\\
1 & 0 & 0 & \cdots & 0\\
0 & 1 & 0 & \cdots & 0\\
\vdots & & \ddots & \ddots & \vdots\\
0 & 0 & \cdots &   1& 0
\end{matrix} \right) 
\,.
\end{equation*} 

In the sequel, we will use so-called \textit{$d$-bonacci numbers}, which are the natural generalizations of the famous Fibonacci numbers.
The sequence of the $d$-bonacci numbers $(D_k)_{k\geq 0}$ is defined by the linear recurrence:
$$
D_k = \sum_{j = 1}^d D_{k-j}\ \text{ for } k \geq d \quad \text{ and } \quad D_k = 2^k \ \text{ for all } k = 0, 1, \ldots, d-1  \,.
$$
Equivalently, the $d$-bonacci numbers can be expressed using the matrix recurrence. 
To simplify the notation we put  $D_{-1} = 1$ and $D_{-k} = 0$ for all $k = 2, \ldots, d$.
We also denote the vector $\vec{D}(n) = (D_n, D_{n-1}, \ldots, D_{n-d+1})^\top$ for all $n \geq -1$.  Then the  recurrence relation for the $d$-bonacci numbers can be rewritten in the following vector form: 
$$ 
\vec{D}(n) = \boldsymbol{M}\vec{D}(n-1) \quad 
\text{ for all } n \in \N \quad  \text{ and } \quad \vec{D}(-1)  = (1, 0, \ldots, 0)^\top = :\vec{e}\,,
$$   
where  $\boldsymbol{M}$ is the matrix of the $d$-bonacci morphism $\tau$.
Obviously, we can write 
\begin{equation} \label{eq:bonacci_numbers}
\vec{D}(n) = \boldsymbol{M}^{n+1} \vec{e}\,. 
\end{equation}

The simple form of the morphism $\tau$ gives us immediately the relation between the consecutive bispecial factors in the $d$-bonacci word $\tt$ (compare with Claim \ref{claim:derived_word_of_preimage}), which allows us to express the lengths of the bispecial factors of $\tt$.

\begin{claim} \label{claim:bonacci_bispecial_factors}
For every $k \geq 1$ the bispecial factors of the $d$-bonacci word $\tt$ fulfil the equation 
$$B_\tt(k) = \tau(B_\tt(k-1))0\,.$$
\end{claim}

\begin{lemma} \label{lemma:dbonacci_bispecials}
For every $k \in \N$ the $k^{th}$ bispecial factor $B_{\tt}(k)$ of the $d$-bonacci word $\tt$ has the length
$$|B_{\bf t}(k)| = \frac{1}{d-1}\sum_{i=0}^{d-1} (d-i)D_{k-i-1} -\frac{d}{d-1}\,, \quad \text{ where $D_j$ is the $j^{th}$ $d$-bonacci number. }$$

\end{lemma}

\begin{proof}
We denote the Parikh vector of the $k^{th}$ bispecial factor $\vec{V}(k)$.
Then using Claim \ref{claim:bonacci_bispecial_factors} and Relation \eqref{eq:parikh_vector_of_image_via_incid_matrix} we may write: 
$$
\vec{V}(k) = \boldsymbol{M}\vec{V}(k-1) + \vec{e} \quad \text{ and so } \quad \vec{V}(k)  = \boldsymbol{M}^{k}\vec{V}(0) + (\boldsymbol{M}^{k-1} + \boldsymbol{M}^{k-2} + \cdots + \boldsymbol{M}^0)\vec{e}\,.
$$
Since $B_\tt(0) = \varepsilon$, it is $\vec{V}(0) = (0, \ldots, 0)^\top$ and 
$$
\vec{V}(k) = (\boldsymbol{M}^{k-1} + \boldsymbol{M}^{k-2} + \cdots + \boldsymbol{M}^0)\vec{e}\,.
$$
If we multiply this equality by the matrix $(\boldsymbol{M}-\boldsymbol{I})$, where $\boldsymbol{I}$ is the identity matrix, we get: 
$$
(\boldsymbol{M}-\boldsymbol{I})\vec{V}(k) = (\boldsymbol{M}^{k} + \boldsymbol{M}^{k-1} + \cdots + \boldsymbol{M}^{1} - \boldsymbol{M}^{k-1} - \cdots - \boldsymbol{M}^{0})\vec{e} = \boldsymbol{M}^k \vec{e} - \vec{e}\,.
$$
Finally, the application of Equation \eqref{eq:bonacci_numbers} gives us:
$$
\vec{V}(k) = (\boldsymbol{M}-\boldsymbol{I})^{-1}\left(\boldsymbol{M}^k \vec{e} - \vec{e}\right) = (\boldsymbol{M}-\boldsymbol{I})^{-1}\left(\vec{D}(k-1) - \vec{e}\right)\,.
$$
Now we can express the length of the $k^{th}$ bispecial factor as:
$$
|B_\tt(k)| = (1, \ldots, 1) \cdot \vec{V}(k) 
= (1, \ldots, 1) \cdot  (\boldsymbol{M}-\boldsymbol{I})^{-1}\left(\vec{D}(k-1) - \vec{e}\right)\,.
$$
It suffices to compute the inverse matrix $(\boldsymbol{M}-\boldsymbol{I})^{-1}$. 
One can verify that it is 
$$
(\boldsymbol{M}-\boldsymbol{I})^{-1} = \frac{1}{d-1}
\left( 
\begin{matrix}
1 & d-1 & d-2 & \cdots & 2 & 1 \\
1 & 0 & d-2 & \cdots & 2 & 1 \\
1 & 0 & -1 & \cdots & 2 & 1 \\
\vdots & \vdots & \vdots & \ddots &  & \vdots \\
1 & 0 & -1 & \cdots & -d+3 & 1 \\
1 & 0 & -1 & \cdots & -d+3 & -d+2
\end{matrix}
\right)
$$
and thus
$(1, \ldots, 1) \cdot  (\boldsymbol{M}-\boldsymbol{I})^{-1}=\frac{1}{d-1}(d, d-1, d-2, \ldots, 1)$. Consequently,

\begin{align*}
|B_\tt(k)| 
&= \frac{1}{d-1}(d, d-1, d-2, \ldots, 1)\vec{D}(k-1) - \frac{d}{d-1} \\
&= \frac{1}{d-1}\sum_{i=0}^{d-1} (d-i)D_{k-i-1} - \frac{d}{d-1}\,.
\end{align*}
\end{proof}

To find the simple expression for the lengths of the return words to $B_\tt(k)$, we state one more auxiliary lemma. 
Let us remind that the $d$-bonacci word has the directive sequence $(0\, 1\, 2 \cdots d-1)^\omega$, as explained in Example \ref{exam:d_bonacci}. 

\begin{lemma}  \label{lemma:dbonacci_return_words}
For the $d$-bonacci word with the directive sequence $(i_n)_{n \in \N} = (0\, 1\, 2 \cdots d-1)^\omega$ and for every integer $k \geq 1$ we have
$$
|\varphi_{i_0} \varphi_{i_1} \cdots \varphi_{i_{k-1}}(i_k)| = |\tau^k(0)| = D_k\,, \quad \text{ where $D_{k}$ is the $k^{th}$ $d$-bonacci number.}
$$

\end{lemma}

\begin{proof}
One can simply verify that $\tau = \varphi_0 \circ P$, where $P$ is a permutation such that $P(a) \equiv a+1 \mod d$ for all $a \in \{0, 1, \ldots, d-1\}$.
It is also easy to realize that $P \circ \varphi_a = \varphi_{b} \circ P$ for every $a,b \in \{0, 1, \ldots, d-1\}$ such that $b \equiv a+1 \mod d$.
These two facts give us  
\begin{multline*}
\tau^k = (\varphi_0 \circ P)^k 
= \varphi_0 \circ (P \circ \varphi_0)^{k-1} \circ P
= \varphi_0 \circ (\varphi_1 \circ P)^{k-1} \circ P 
= \varphi_0 \varphi_1 \circ (P \circ \varphi_1)^{k-2} \circ P^2  = \cdots \\ 
=  \varphi_{j_0}\varphi_{j_1} \cdots \varphi_{j_{k-1}}P^k\,,
\end{multline*}
where $j_n \in \mathcal{A}$ and $j_n \equiv n \mod d$.
But since the sequence $(j_n)_{n \in \N}$ is exactly the directive sequence of the $d$-bonacci word, i.e., $j_n = i_n$ for every $n \in \N$, we may conclude that  
$$
\tau^k(0)
= \varphi_{j_0}\varphi_{j_1} \cdots \varphi_{j_{k-1}}P^k(0) 
= \varphi_{j_0}\varphi_{j_1} \cdots \varphi_{j_{k-1}}(j_k)
= \varphi_{i_0} \varphi_{i_1} \cdots \varphi_{i_{k-1}}(i_k)\,.
$$

It remains to prove that $|\tau^k(0)| = D_{k}$. We will prove that both sequences $(|\tau^n(0)|)_{n \in \N}$ and $(D_n)_{n \in \N}$ fulfil the same linear recurrence with the same initial conditions.
In fact, we will show that for every $a \in \mathcal{A}$ the following equalities hold:  
\begin{equation} \label{eq:tau_recurrence}
|\tau^{k}(a)| = \sum_{j = 1}^{d-a} |\tau^{k-j}(0)| \ \text{ for all } k \geq d-a \quad \text{ and } \quad
|\tau^k(a)| = 2^k \ \text{ for all } k = 0, \ldots, d-a-1  \,.
\end{equation}
We will proceed by induction on $k$.
Simple computations verify the initial conditions. 
Now we suppose that the equality is true for $k-1$ and every letter $a \in \mathcal{A}$ and we prove that it is true also for $k$.
If $a = d-1$, it is clear since $\tau^{k}(d-1) = \tau^{k-1}(0)$.
If $a \neq d-1$, we rewrite as follows:   
$$
|\tau^{k}(a)| = |\tau^{k-1}(0)| + |\tau^{k-1}(a+1)| =  |\tau^{k-1}(0)| + \sum_{j = 1}^{d-a-1} |\tau^{k-1-j}(0)| = \sum_{j = 1}^{d-a} |\tau^{k-j}(0)|\,.
$$

If we consider the relations \eqref{eq:tau_recurrence} for the letter $a = 0$, we get exactly the same recurrence as in the case of $d$-bonacci numbers. 
Thus these two sequences are the same and $|\tau^k(0)| = D_k$.
\end{proof}

\begin{theorem}
Let $\tt$ be the $d$-bonacci word and let $n, k$ be positive integers such that 
$$
\frac{1}{d-1}\sum_{i=0}^{d-1} (d-i)D_{k-i-2} -\frac{d}{d-1} < n \leq  \frac{1}{d-1}\sum_{i=0}^{d-1} (d-i)D_{k-i-1} -\frac{d}{d-1}\,.
$$
Then
$$
nr\mathcal{C}_\tt(n) = D_{k+1} -1 - \frac{1}{d-1}\sum_{i=0}^{d-1}(d-i)D_{k-i-1} +\frac{d}{d-1} + n\,.
$$
\end{theorem}

\begin{proof}
It follows directly from Theorem \ref{thm:nonrepetitive_complexity_arnoux_rauzy}. 
It suffices to replace the lengths of the bispecial factors by the expressions from Lemma \ref{lemma:dbonacci_bispecials} and determine the value of 
$$
|\varphi_{i_0}\varphi_{i_1} \cdots \varphi_{i_{k-1}}(i_k a)|\,,$$
where $a \in \mathcal{A}$ is any letter such that  $S_a(k) = \inf \{S_b(k) : b  \in \mathcal{A}, b \neq i_k \}$.
Since $\tt$ has the directive sequence $(i_n)_{n \geq 0}$ given by $i_n \equiv n \mod d$, it is easy to realize that the desired letter $a$ is the letter $i_{k+1}$ (note that for $k < d-2$ there are also other possible choices of $a$). 
Then using Lemma \ref{lemma:dbonacci_return_words} we get
$$
|\varphi_{i_0}\varphi_{i_1} \cdots \varphi_{i_{k-1}}(i_k a)| 
= |\varphi_{i_0}\varphi_{i_1} \cdots \varphi_{i_{k-1}}(i_k i_{k+1})| 
= |\varphi_{i_0}\varphi_{i_1} \cdots \varphi_{i_{k}}(i_{k+1})| = |\tau^{k+1}(0)| = D_{k+1}\,.$$
\end{proof}

\begin{theorem} \label{thm:bonacci_initial_nonrepetitive_complexity}
Let $\tt$ be the $d$-bonacci word and let $n, k $ be positive integers such that 
$$
\frac{1}{d-1}\sum_{i=0}^{d-1} (d-i)D_{k-i-2} -\frac{d}{d-1} < n \leq  \frac{1}{d-1}\sum_{i=0}^{d-1} (d-i)D_{k-i-1} -\frac{d}{d-1}\,.
$$
Then $inr\mathcal{C}_\tt(n) = D_k \,.$
\end{theorem}

\begin{proof} 
It follows directly from Theorem \ref{thm:initial_nonrepetitive_complexity_arnoux_rauzy}. 
It suffices to realize that by Lemma \ref{lemma:dbonacci_bispecials} we know the lengths of the bispecial factors of $\tt$ and by Lemma \ref{lemma:dbonacci_return_words} we have $|\varphi_{i_0}\varphi_{i_1} \cdots \varphi_{i_{k-1}}(i_{k})| =  |\tau^{k}(0)| = D_k$. 
\end{proof}

Note that for $d=2$ and $d= 3$ the previous theorem gives 
 the results  stated in \cite{NiRa16}  as Theorems  10 and 16.

\begin{corollary}  
Let ${\bf f}$  and $\tt$ be  the Fibonacci and the Tribonacci word, respectively. 

\begin{itemize}
\item Let $n, k$ be positive integers such that $F_k -2 < n \leq F_{k+1} -2$. Then $inr\mathcal{C}_{\bf f}(n) = F_k$, where $F_k$ is the $k^{th}$ Fibonacci number.  
\item Let $n, k$ be positive integers such that 
$\frac{T_k + T_{k-2} -3}{2} < n \leq \frac{T_{k+1} + T_{k-1} -3}{2}\,.$ 
Then $inr\mathcal{C}_\tt(n) = T_k$, where $T_k$ is the $k^{th}$ Tribonacci number.  
\end{itemize}
\end{corollary}


\begin{thebibliography}{00}

\bibitem{ArRo91} 
P. Arnoux, G. Rauzy: Repr\'esentation g\'eom\'etrique de suites de complexit\'e 2n + 1. Bull.  Soc.  Math.  France 119  (2) (1991), 199--215.

\bibitem{BaPeSt10} 
L'. Balkov\'a, E. Pelantov\'a, \v S. Starosta: Sturmian Jungle (or Garden?) on Multiliteral Alphabets. RAIRO - Theoret. Inf. Appl. 44 (2010), 443--470. 

\bibitem{BeCaSt13}
V. Berth\'e, J. Cassaigne, W. Steiner: Balance properties of Arnoux--Rauzy words. Int. J.  of Algebra and Computation, Vol. 23, No. 04,  (2013),  689--703. 

\bibitem{7people15} 
V.  Berth\'e, C. De Felice, F. Dolce, J. Leroy, D. Perrin, Ch. Reutenauer, G. Rindone: Acyclic, connected and tree sets. Monatshefte für Mathematik 176 (4) (2015).


\bibitem{BuKi19}
Y. Bugeaud, D. H. Kim: A new complexity function, repetitions in Sturmian words, and irrationality exponents of Sturmian numbers. Trans. Amer. Math. Soc. 371 (2019), 3281--3308.

\bibitem{Ca99} 
J.Cassaigne: Limit values of the recurrence quotient of Sturmian sequences. Theoretical Computer Science 218 (1999), 3--12. 

\bibitem{CaCh}
J. Cassaigne, N. Chekhova: Fonctions de r\'ecurrence des suites d' Arnoux--Rauzy et r\'epons\' e a une question de Morse et Hedlund. Annales de l'Institut Fourier 56 (2006) (7), 2249--2270.

\bibitem{JuVu}
J. Justin, L. Vuillon: Return words in Sturmian and episturmian words. RAIRO-Theoret. Inf. Apppl. 34 (2000), 343--356.

\bibitem{Med}
K. Medkov\'{a}: Derived sequences of Arnoux--Rauzy sequences. In: R.
Mercas and D. Reidenbach (eds.), Proceedings WORDS 2019, Lecture Notes in Computer Science 11682 (2019), Springer, pp. 251--263. 

\bibitem{Fo02} 
N. Pytheas Fogg: Substitutions in dynamics, arithmetics and combinatorics, Lecture Notes in Math. 1794 (2002), Springer-Verlag.

\bibitem{GlJu09} 
A. Glen, J. Justin: Episturmian words:  a survey.
Theor. Inform. Appl. 43 (3) (2009), 403--442. 

\bibitem{MoHe40} 
M.  Morse,  G. A.  Hedlund: Symbolic  dynamics  II:  Sturmian  trajectories.  Amer.  J.  Math.  61  (1940), 1--42. 


\bibitem{Mo12} 
T. K. S. Moothathu: Eulerian entropy and non-repetitivity subword complexity. Theoretical Computer Science 420 (2012), 80--88. 

\bibitem{NiRa16}
J. Nicholson, N.Rampersad: Initial non-repetitivity complexity of infinite words. Discrete Appl. Math. 208 (C) (2016), 114--122. 


\bibitem{RiZa00}
R. N. Risley and L. Q. Zamboni:  
A generalization of Sturmian sequences:
Combinatorial structure and transcendence. Acta Arith. 
95 (2000), 167--184. 

\end{thebibliography}
\end{document}